\documentclass[a4paper, oneside, reqno, 12pt]{amsart}
\usepackage[english]{babel}
\usepackage[utf8]{inputenc}
\usepackage{amsmath, amssymb, amsfonts}
\usepackage{mathtools}
\usepackage{fullpage}
\usepackage{comment}
\usepackage{textcomp, cmap}

\usepackage[usenames,dvipsnames]{xcolor}

\usepackage[backref=page]{hyperref}
\hypersetup{
	colorlinks   = true, 
	urlcolor     = blue, 
	linkcolor    = Blue, 
	citecolor   = Green 
}
\usepackage[capitalise, nameinlink, noabbrev]{cleveref}
\usepackage{autonum}
\usepackage{caption}
\usepackage[font = small]{subcaption}
\usepackage{pgf, tikz}
\usetikzlibrary{arrows, angles, arrows.meta, matrix}
\usepackage{float} 
\usepackage{amsthm} 
\theoremstyle{plain}
\newtheorem{theorem}{Theorem}
\newtheorem{conjecture}[theorem]{Conjecture}

\newtheorem{lemma}[theorem]{Lemma}
\newtheorem{corollary}[theorem]{Corollary}
\theoremstyle{definition}
\newtheorem{problem}[theorem]{Problem}

\theoremstyle{remark}
\newtheorem*{remark*}{Remark}

\definecolor{ffqqqq}{rgb}{1.,0.,0.}

\DeclareMathOperator{\cone}{cone}

\newcommand{\ncrlsetcard}[1]{|DJ_\ell(#1)|}

\title[Geometric graphs]{{Disjoint edges in geometric graphs}}
\author[N.~Chernega, A.~Polyanskii, R.~Sadykov]{{Nikita~Chernega, Alexandr~Polyanskii, Rinat~Sadykov}}

\address{Nikita Chernega, \newline\hphantom{iii} Moscow Institute of Physics and Technology, Institutskiy per. 9, Dolgoprudny, Russia 141700
}
\email{\href{mailto:chernega_nikita@mail.ru}{chernega\_nikita@mail.ru}}

\address{Alexandr Polyanskii,
\newline\hphantom{iii} Moscow Institute of Physics and Technology, Institutskiy per. 9, Dolgoprudny, Russia 141700
}
\email{\href{mailto:alexander.polyanskii@yandex.ru}{alexander.polyanskii@yandex.ru}}
\urladdr{\url{http://polyanskii.com}}

\address{Rinat Sadykov,
\newline\hphantom{iii} Moscow Institute of Physics and Technology, Institutskiy per. 9, Dolgoprudny, Russia 141700
}
\email{\href{mailto:sadykov@phystech.edu}{sadykov@phystech.edu}}


\begin{document}

\thispagestyle{empty}

\begin{abstract}
A \textit{geometric graph} is a graph drawn in the plane so that its vertices and edges are represented by points in general position and straight line segments, respectively. A vertex of a geometric graph is called \textit{pointed} if it lies outside of the convex hull of its neighbours.  We show that for a geometric graph with \(n\) vertices and \(e\) edges there are at least \(\frac{n}{2}\binom{2e/n}{3}\) pairs of disjoint edges provided that \(2e\geq n\) and all the vertices of the graph are pointed. Besides, we prove that if any edge of a geometric graph with \(n\) vertices is disjoint from at most \( m \) edges, then the number of edges of this graph does not exceed \(n(\sqrt{1+8m}+3)/4\) provided that \(n\) is sufficiently large.

These two results are tight for an infinite family of graphs.
\end{abstract}

\maketitle

\section{Introduction}

A \textit{geometric graph} \(G\) is a graph drawn in the plane by (possibly crossing) straight line segments, that is, its vertex set \(V(G)\) is a set of points in general position in the plane and its edge set \(E(G)\) is the set of straight line segments with endpoints belonging to \(V(G)\). One of the classical problems on geometric graphs is a question raised by Avital and Hanani~\cite{avital1966graphs}, Kupitz~\cite{Kupitz79}, Erd\H os and Perles: For positive integers \(k\) and \(n\), determine the smallest \(e_k(n)\) such that any geometric graph with \(n\) vertices and \(m>e_k(n)\) edges contains \(k+1\) pairwise disjoint edges.

By results of Hopf and Pannwitz~\cite{hopf1934} and Erd\H os~\cite{Erdos46}, we know that \(e_1(n)=n\). The upper bound for \(e_2(n)\) was studied in papers of Alon and Erd\H os~\cite{AlonErdos89}, Goddard, Katchalski, and Kleitman~\cite{GoddardKatchalskiKleitman96}, M\'esz\'aros~\cite{meszaros98}. The current best upper bound \(e_2(n)\leq \lceil 5n/2\rceil\) was proved by \v{C}ern\'{y}~\cite{Cerny05}. This bound is tight up to additive constant: Perles found an example showing that \(e_2(n)\geq \lfloor 5n/2 \rfloor -3\). Also, in~\cite{GoddardKatchalskiKleitman96} it was shown that \({7n/2-6\leq e_3(n)\leq 10n}\). For \(k\leq n/2\), Kupitz~\cite{Kupitz79} proved the lower bound \(e_k(n)\geq kn\), and later T\'oth and Valtr~\cite{TothValtr99} improved it: \({e_k(n)\geq 3(k-1)n/2-2k^2}\). 
Using Dilworth's theorem, Pach and  T\"or\H{o}csik~\cite{PachTorocsik94} found a beautiful proof of the upper bound~\(e_k(n)\leq k^4 n\). Later, this bound was refined in~\cite{TothValtr99}, and the current best upper bound \(e_k(n)\leq 256k^2n\) belongs to T\'oth~\cite{Toth00}; see also~Theorem~1.11 in~\cite{felsner2012geometric}. Another interesting result about disjoint edges of a \textit{convex graph} is due to Kupitz~\cite{Kupitz79}. Recall that \textit{a convex graph} is a geometric graph whose vertices are in convex position. He proved that if a convex graph on \( n \) vertices has no \(k+1\) pairwise disjoint edge, then its number of edges does not exceed \(kn\) provided \(n\geq 2k+1\). Keller and Perles~\cite{keller2012smallest} studied the case $n=2k+2$ and gave the exact characterization of the extremal configurations (that is, with the maximum number of edges); see~Theorem~1.5 in their paper, which is stated in terms of the so-calling blocking sets. For further reading, we refer the interested readers to the survey of Pach~\cite{pach2013beginnings} on geometric and topological graphs.

All these classical results are about \textit{how many edges in a geometric graph on \(n\) vertices guarantee \( k+1 \) disjoint edges}. Motivated by them, we focus on the case  \( k=1 \) and study \textit{how many edges in a geometric graph on \(n\) vertices guarantee a lot of pairs of disjoint edges}.

To state the concrete problems, we introduce the following notation. For a geometric graph \( G \), denote by \(DJ(G)\) the set of pairs of disjoint edges of geometric graphs. For an edge \(uv\in E(G)\), let \(DJ(uv)\) be the set of edges in \(G\) disjoint from \(uv\). Clearly, we have the equality
\[
|DJ(G)| = \frac12 \sum_{uv\in E(G)} |DJ(uv)|.
\]

These are the main problems of the paper.
\begin{problem}
    For integers \(n>0\) and \(m\geq0\), determine the greatest number \(e(n,m)\) such that a geometric graph \(G\) on \(n\) vertices has at most \(e(n,m)\) edges provided that \(|DJ(uv)|\leq m\) for any \(uv\in E(G)\).
\end{problem}

\begin{problem}
    For positive integers \(n\) and \(e\), determine the smallest number \(dj(n,e)\) such that for any geometric graph \(G \) with \(n\) vertices and \(e\) edges, we have \(|DJ(G)|\geq dj(n,e)\).
\end{problem}

In particular, we study these problems for  \textit{pointed graphs}. 
To define these graphs, recall that the \textit{neighbourhood} \(N(v)\) of a vertex \(v\in V(G)\) is the set of vertices adjacent to~\( v \). A vertex is called \textit{pointed} if it lies outside of the convex hull of its neighbourhood, otherwise, it is called \textit{cyclic}. A \textit{pointed graph} is a geometric graph such that any of its vertices is pointed; see Figure~\ref{figure: mathcal L and mathcal R}, where a pointed graph is drawn. Clearly, any convex graph is pointed as well.

Also, for positive integer $k$ and any $\alpha\in \mathbb R$, we use the standard notation of binomial coefficient
\[
    \binom{\alpha}{k} : = \frac{\alpha (\alpha-1) \dots (\alpha - k + 1)}{k!}.
\]
The goal of this paper is to prove the following two theorems.
\begin{theorem}\label{theorem:main1}
Let \(m\) be a non-negative integer and \(G\) be a geometric graph such that \( |DJ(uv)|\leq m\) for any edge \( uv\in E(G) \). Then
\[
|E(G)|\leq \max\left(|V(G)|\left(\sqrt{1+8m}+3\right)/4, |V(G)| + 3m-1\right).
\] 

\end{theorem}

\begin{theorem}\label{theorem:main2}
For a pointed graph \( G \) with \( 2|E(G)|\geq |V(G)| \), we have
\[
|DJ(G)| \geq \frac{|V(G)|}{2} \cdot \binom{d(G)}{3},
\]
where \(d(G)=2|E(G)|/|V(G)|\) is the average degree of \(G\).
\end{theorem}

Note that Theorem~\ref{theorem:main1} is a strengthening of a result mentioned above.
\begin{theorem}[Hopf and Pannwitz~\cite{hopf1934}, Erd\H{o}s~\cite{Erdos46}]\label{erdos}
If every edge of a geometric graph \( G \) intersects all other edges of \(G\), then
\(|E(G)|\leq |V(G)|\).
\end{theorem}

The rest of the paper is organized as follows. In Section~\ref{section:preliminaries}, we prove auxiliary lemmas. In Sections~\ref{section:proof1} and~\ref{section:proof2}, we prove Theorems~\ref{theorem:main1} and~\ref{theorem:main2}, respectively. In Section~\ref{section:discussion}, we show that these theorems are tight for an infinite family of graphs. Also, in this section, we discuss open problems related to strengthenings of our main results.

\section{Preliminaries}
\label{section:preliminaries}

Throughout the rest of the paper, we denote by \( n \) and \(e\) the number of vertices and edges of a geometric graph \(G\), respectively, that is, \(n=|V(G)|\) and \(e=|E(G)|\). In this section, we additionally assume that \( G \) is a pointed graph such that every vertex has degree at least~\(2\).

For distinct points \(x,y,z\in \mathbb R^2\) in general position, by the \textit{oriented angle} \( \angle xyz \) we mean \( \alpha\) in the range \( (-\pi, \pi) \) such that the rotation by the angle \( \alpha \) around \( y \) maps the ray \( yx \) to the ray \( yz \). Here we assume that if \(\alpha\) is positive, then the corresponding rotation is in the counterclockwise direction, otherwise, it is in the clockwise direction; see Figure~\ref{figure: mathcal L and mathcal R}. By this definition, we have \( \angle xyz=-\angle zyx\). 

For a vertex \( v\in V(G)\), choose \( x,y\in N(v)\) such that
\begin{equation}
    \label{equation:max_angle}
\angle xvy =\max\left\{\angle avb: a,b \in N(v)\right\}.
\end{equation}
Set \( \ell_v: = y \) and \( r_v:=x \). The edges \( v\ell_v \) and \( vr_v \) are called the \textit{leftmost} and \textit{rightmost edges} of \( v \), respectively. Analogously, we call the vertices \( \ell_v \) and \( r_v \) the \textit{leftmost} and \textit{rightmost neighbours} of \( v \), respectively.
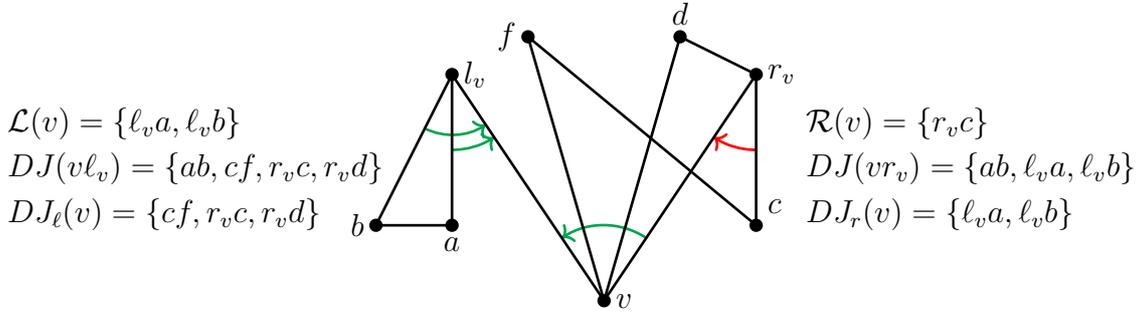
\begin{figure}[h]
						\begin{tikzpicture}[line cap=round,line width = 1pt, x = 1cm, y = 1cm]
							\coordinate (v) at (0,0);
							\coordinate (d) at (1,3.5);
							\coordinate (e) at (-1,3.5);
							\coordinate (r) at (2,3);
							\coordinate (l) at (-2,3);
							\coordinate (c) at (2,1);
							\coordinate (a) at (-2,1);
							\coordinate (b) at (-3,1);
							\draw (v) node  [right]{\(v\)};
							\draw (d)  node  [above]{\( d \)};
							\draw (r)  node  [right]{\( r_v \)};
							\draw (l)  node  [right]{\( l_v \)};
							\draw (c)  node  [above right]{\( c \)};
							\draw (b)  node  [left]{\( b \)};
							\draw (a)  node  [below]{\( a \)};
							\draw (e)  node  [left]{\( f \)};
							
							\filldraw (v) circle (2pt);
							\filldraw (d)  circle (2pt);
							\filldraw (l) circle (2pt);
							\filldraw (r)  circle (2pt);
							\filldraw (c)  circle (2pt);
							\filldraw (a)  circle (2pt);
							\filldraw (b)  circle (2pt);
							\filldraw (e)  circle (2pt);
							
							\draw pic[draw=red, <-, angle radius=1cm]
							{angle=v--r--c};
							\draw pic[draw=Green, ->, angle radius=1cm]
							{angle=a--l--v};
							\draw pic[draw=Green, ->, angle radius=0.8cm]
							{angle=b--l--v};
							\draw pic[draw=Green, ->, angle radius=1cm]
							{angle=r--v--l};
							\draw (v)--(d);
							\draw (v)--(l);
							\draw (v)--(r);
							\draw (r)--(c);
							\draw (l)--(a);
							\draw (l)--(b);
							\draw (d)--(r);
							\draw (e)--(v);
							\draw (a)--(b);
							\draw (e)--(c);
							
							\draw (-8, 2.35) node [right]{$\mathcal{L}(v) = \{\ell_va, \ell_vb\}$};
							\draw (-8, 1.75) node [right]{$DJ(v\ell_v) = \{ab, cf, r_vc, r_vd\}$};
							\draw (-8, 1.15) node [right]{$DJ_\ell(v) = \{cf, r_vc, r_vd\}$};
							
							\draw (2.5, 2.35) node [right]{\(\mathcal{R}(v) = \{r_vc\}\)};
							\draw (2.5, 1.75) node [right]{\(DJ(vr_v) = \{ab, \ell_va, \ell_vb\}\)};
							\draw (2.5, 1.15) node [right]{\(DJ_r(v) = \{\ell_va, \ell_vb\}\)};
							
							\end{tikzpicture}
					
					\caption{The green angles
					are positive and the red angle
					is negative.}
				\label{figure: mathcal L and mathcal R}	
					\end{figure}

It turns out that in the case of pointed graphs it is enough to focus only on the edges disjoint from leftmost and rightmost edges.
To formalize this idea, we introduce the following important notation that we will use instead of \(DJ(uv)\). 

For \( v \in V(G) \), let \( DJ_\ell(v) \) be the set of edges incident to one of the vertices of \( N(v) \) and disjoint from the leftmost edge \( v\ell_v \). Analogously, let \(DJ_r(v) \) be the set of edges incident to one of the vertices of \( N(v) \) and disjoint from the rightmost edge \( vr_v \); see Figure~\ref{figure: mathcal L and mathcal R}.  Clearly, \( DJ_\ell(v)\subseteq DJ(v\ell_v)\) and \(DJ_r(v)\subseteq DJ(vr_v) \), and thus, we have \(|DJ(v\ell_v)| \geq |DJ_\ell(v)|\) and \( |DJ(vr_v)| \geq |DJ_r(v)|\).  

For \(v\in V(G)\), denote by \(\mathcal L(v) \) the set of edges \(\ell_vx\in E(G)\) such that the angle \(\angle x\ell_v v\) is positive; see Figure~\ref{figure: mathcal L and mathcal R}. Equivalently, the edge \( \ell_v x\) belongs to \(\mathcal L(v) \) if and only if the ray \(\ell_v x\) shares with the affine convex cone \(v+ \cone\{\ell_v-v,r_v -v\} \) only the point~\(\ell_v\). Clearly, the edges from \( \mathcal L(v) \) are disjoint from \(vr_v\) and \(\mathcal L(v)\subset DJ_r(v)\). 

Analogously, denote by \(\mathcal R(v)\) the set of edges \(r_vv\in E(G)\) such that the angle \(\angle xr_vv\) is negative; see Figure~\ref{figure: mathcal L and mathcal R}. Clearly, the edges from \( \mathcal R(v) \) are disjoint from \( v\ell_v \) and \( \mathcal R(v) \subset DJ_\ell(v) \). Remark that if \(\ell_v x\in \mathcal L(v)\) or \(r_vx\in \mathcal R(v)\), then \(x\not\in N(v)\).

\begin{lemma}
\label{lemma:simple?}
For any vertex \( v\in V(G) \), we have
\begin{equation}
\label{equation: first lemma}
    | DJ_\ell (v) | + | DJ_r (v) | \geq \sum_{w\in N(v)\setminus\{\ell_v, r_v\}} (\deg w -1) + |\mathcal L(v)| + |\mathcal R(v)|,
\end{equation}
where \( \deg w \) is the degree of vertex \(w\in V(G)\).
\end{lemma}
\begin{proof}

To prove this lemma, we apply the so-called discharging method. Let us assign a charge to every edge \(wt\in E(G)\) as follows:
\begin{itemize}
    \item[1.] If the vertex \(v\) coincides with \(w\) or \(t\), then the charge of \(wt\) is 0.
    \item[2.] If \(wt\in \mathcal L(v)\cup \mathcal R(v)\), then the charge of \(wt\) is 1.
    \item[3.] If \(wt\not\in \mathcal L(v)\cup \mathcal R(v)\) and \(v\) is distinct from \(w\) and \(t\), then the charge of \(wt\) is 
    \[|\{w,t\}\cap N(v)\setminus\{\ell_v,r_v\}|,\]
    which can be equal to \(0,1,\) or \(2\); see the Figure~\ref{figure: cases in Lemma 6}.
\end{itemize}

		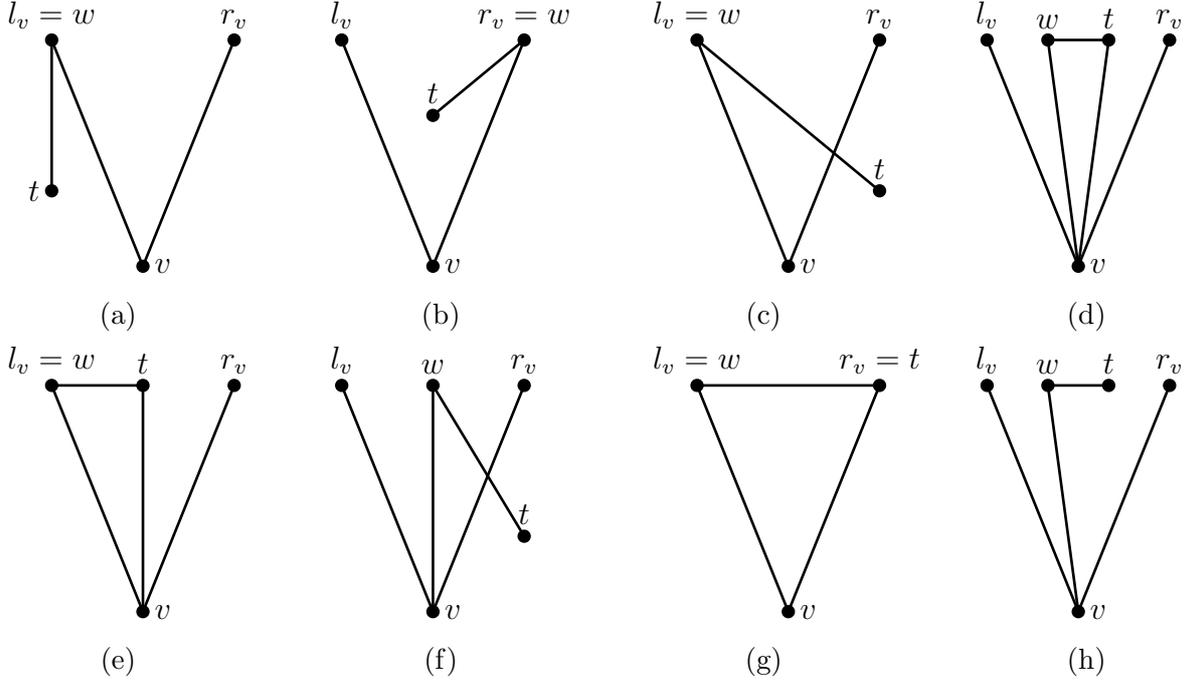
\begin{figure}[h]
				\begin{subfigure}[b]{0.2\textwidth}
					\begin{tikzpicture}[line cap=round,line width = 1pt, x = 0.8cm, y = 1cm]
						\coordinate (v) at (0,0);
						\coordinate (r) at (1.5,3);
						\coordinate (l) at (-1.5,3);
						\coordinate (t) at (-1.5,1);
						
						\draw (v) node  [right]{\(v\)};
						\draw (r)  node  [above]{\( r_v \)};
						\draw (l)  node  [above]{\( l_v=w \)};
						\draw (t)  node  [left]{\( t \)};
						
						\filldraw (v) circle (2pt);
						\filldraw (l) circle (2pt);
						\filldraw (r)  circle (2pt);
						\filldraw (t)  circle (2pt);
						\draw (v) -- (l);
						\draw (v) -- (r);
						\draw (l) -- (t);
						
				\end{tikzpicture}
			\caption{}
			\label{figure1:a}
				\end{subfigure}
			\hfill
			\begin{subfigure}[b]{0.2\textwidth}
					\begin{tikzpicture}[line cap=round,line width = 1pt, x = 0.8cm, y = 1cm]
						\coordinate (v) at (0,0);
						\coordinate (r) at (1.5,3);
						\coordinate (l) at (-1.5,3);
						\coordinate (t) at (0,2);
						
						\draw (v) node  [right]{\(v\)};
						\draw (r)  node  [above]{\( r_v=w \)};
						\draw (l)  node  [above]{\( l_v \)};
						\draw (t)  node  [above]{\( t \)};
						
						\filldraw (v) circle (2pt);
						\filldraw (l) circle (2pt);
						\filldraw (r)  circle (2pt);
						\filldraw (t)  circle (2pt);
						\draw (v) -- (l);
						\draw (v) -- (r);
						\draw (r) -- (t);
						
				\end{tikzpicture}
			\caption{}
			\label{figure1:b}
				\end{subfigure}
			\hfill
			\begin{subfigure}[b]{0.2\textwidth}
				\begin{tikzpicture}[line cap=round,line width = 1pt, x = 0.8cm, y = 1cm]
					\coordinate (v) at (0,0);
					\coordinate (r) at (1.5,3);
					\coordinate (l) at (-1.5,3);
					\coordinate (t) at (1.5,1);
					
					\draw (v) node  [right]{\(v\)};
					\draw (r)  node  [above]{\( r_v \)};
					\draw (l)  node  [above]{\( l_v=w \)};
					\draw (t)  node  [above]{\( t \)};
					
					\filldraw (v) circle (2pt);
					\filldraw (l) circle (2pt);
					\filldraw (r)  circle (2pt);
					\filldraw (t)  circle (2pt);
					\draw (v) -- (l);
					\draw (v) -- (r);
					\draw (l) -- (t);
					
				\end{tikzpicture}
				\caption{}
				\label{figure1:c}
			\end{subfigure}
		\hfill
		\begin{subfigure}[b]{0.2\textwidth}
			\begin{tikzpicture}[line cap=round,line width = 1pt, x = 0.8cm, y = 1cm]
			\coordinate (v) at (0,0);
			\coordinate (r) at (1.5,3);
			\coordinate (l) at (-1.5,3);
			\coordinate (t) at (0.5,3);
			\coordinate (w) at (-0.5,3);
			
			\draw (v) node  [right]{\(v\)};
			\draw (r)  node  [above]{\( r_v \)};
			\draw (l)  node  [above]{\( l_v\)};
			\draw (t)  node  [above]{\( t \)};
			\draw (w)  node  [above]{\( w \)};
			
			\filldraw (v) circle (2pt);
			\filldraw (l) circle (2pt);
			\filldraw (r)  circle (2pt);
			\filldraw (t)  circle (2pt);
			\filldraw (w)  circle (2pt);
			\draw (v) -- (l);
			\draw (v) -- (r);
			\draw (w) -- (t);
			\draw (t) -- (v);
			\draw (w) -- (v);

			\end{tikzpicture}
			\caption{}
			\label{figure1:d}
		\end{subfigure}
	\hfill
	\begin{subfigure}[b]{0.2\textwidth}
		\begin{tikzpicture}[line cap=round,line width = 1pt, x = 0.8cm, y = 1cm]
			\coordinate (v) at (0,0);
			\coordinate (r) at (1.5,3);
			\coordinate (l) at (-1.5,3);
			\coordinate (t) at (0,3);
			
			\draw (v) node  [right]{\(v\)};
			\draw (r)  node  [above]{\( r_v \)};
			\draw (l)  node  [above]{\( l_v=w \)};
			\draw (t)  node  [above]{\( t \)};
			
			\filldraw (v) circle (2pt);
			\filldraw (l) circle (2pt);
			\filldraw (r)  circle (2pt);
			\filldraw (t)  circle (2pt);
			\draw (v) -- (l);
			\draw (v) -- (r);
			\draw (l) -- (t);
			\draw (t) -- (v);
			
		\end{tikzpicture}
		\caption{}
		\label{figure1:e}
	\end{subfigure}
\hfill
\begin{subfigure}[b]{0.2\textwidth}
	\begin{tikzpicture}[line cap=round,line width = 1pt, x = 0.8cm, y = 1cm]
		\coordinate (v) at (0,0);
		\coordinate (r) at (1.5,3);
		\coordinate (l) at (-1.5,3);
		\coordinate (t) at (1.5,1);
		\coordinate (w) at (0,3);
		
		\draw (v) node  [right]{\(v\)};
		\draw (r)  node  [above]{\( r_v \)};
		\draw (l)  node  [above]{\( l_v \)};
		\draw (t)  node  [above]{\( t \)};
		\draw (w)  node  [above]{\( w \)};
		
		\filldraw (v) circle (2pt);
		\filldraw (l) circle (2pt);
		\filldraw (r)  circle (2pt);
		\filldraw (t)  circle (2pt);
		\filldraw (w)  circle (2pt);
		\draw (v) -- (l);
		\draw (v) -- (r);
		\draw (w) -- (t);
		\draw (w) -- (v);
		
	\end{tikzpicture}
	\caption{}
		\label{figure1:f}
	\end{subfigure}
	\hfill
	\begin{subfigure}[b]{0.2\textwidth}
		\begin{tikzpicture}[line cap=round,line width = 1pt, x = 0.8cm, y = 1cm]
		\coordinate (v) at (0,0);
		\coordinate (r) at (1.5,3);
		\coordinate (l) at (-1.5,3);
		\coordinate (t) at (1.5,1);
		\coordinate (w) at (0,3);
		
		\draw (v) node  [right]{\(v\)};
		\draw (r)  node  [above]{\( r_v = t\)};
		\draw (l)  node  [above]{\( l_v = w\)};
		
		\filldraw (v) circle (2pt);
		\filldraw (l) circle (2pt);
		\filldraw (r)  circle (2pt);
		\draw (v) -- (l);
		\draw (v) -- (r);
		\draw (r) -- (l);

		\end{tikzpicture}
		\caption{}
		\label{figure1:g}
	\end{subfigure}
	\hfill
	\begin{subfigure}[b]{0.2\textwidth}
		\begin{tikzpicture}[line cap=round,line width = 1pt, x = 0.8cm, y = 1cm]
		\coordinate (v) at (0,0);
		\coordinate (r) at (1.5,3);
		\coordinate (l) at (-1.5,3);
		\coordinate (t) at (0.5,3);
		\coordinate (w) at (-0.5,3);
		
		\draw (v) node  [right]{\(v\)};
		\draw (r)  node  [above]{\( r_v \)};
		\draw (l)  node  [above]{\( l_v\)};
		\draw (w)  node  [above]{\( w \)};
		\draw (t)  node  [above]{\( t \)};
		
		\filldraw (v) circle (2pt);
		\filldraw (l) circle (2pt);
		\filldraw (r)  circle (2pt);
		\filldraw (w)  circle (2pt);
		\filldraw (t)  circle (2pt);
		\draw (v) -- (l);
		\draw (v) -- (r);
		\draw (w) -- (v);
		\draw (w) -- (t);
			
		\end{tikzpicture}
		\caption{}
		\label{figure1:h}
	\end{subfigure}
	
\caption{In cases (b), (c), and (g), the charge of $wt$ is 0, in cases (a), (e), (f), and (h), the charge $wt$ is 1, and in case (d) the charge of $wt$ is 2.}
	\label{figure: cases in Lemma 6}
\end{figure}

Notice that the charge of an edge is well-define. Moreover, it equals~2 if and only the edge connects two neighbours of \(v\) distinct from its rightmost and leftmost neighbours. Thus we conclude that the sum of charges of all edges equals the right-hand side of the desired inequality. 

Notice that an edge has a positive charge only if it connects one of the neighbours of $v$ distinct from $\ell_v$ and $r_v$. Hence it is enough to show that if an edge \(wt\in E(G)\) has charge 1 or 2, then it is disjoint from one or two of the edges \(vr_v\) and \(v\ell_v\), respectively. There are the following possible cases:
    \begin{itemize}
        \item[1.] If \(wt\) belongs to \(\mathcal L(v)\) or \( \mathcal R(v) \) then it has charge 1 and is disjoint from \(vr_v\) or \(v\ell_v\), respectively. See Figure~\ref{figure1:a}.
        \item[2.] If the edge \( wt \) connects a vertex in \(N(v)\setminus\{r_v,\ell_v\}\) with a vertex in \[V(G)\setminus (\{v\}\cup N(v)\setminus\{r_v,\ell_v\}),\] then \(wt\) has charge 1 and is disjoint from at least one of the edges \(v\ell_v\) or \(vr_v\). See Figures~\ref{figure1:e}, ~\ref{figure1:f} and~\ref{figure1:h}.
        \item[3.] If the edge \(wt\) connects two vertices from \(N(v)\setminus\{r_v,\ell_v\}\), then \(wt\) has charge 2 and lies in the interior of the affine convex cone \(v+\cone \{ r_v- v, \ell_v-v\}\), and thus, it is disjoint from the edges \(v\ell_v\) and \(vr_v\). See Figure~\ref{figure1:d}.
    \end{itemize}
    
Since any other edge has charge 0, we are done; see Figures~\ref{figure1:c}, ~\ref{figure1:b} and ~\ref{figure1:g}.
\end{proof}

For \( v \in V(G) \), denote by \( \alpha_{\ell}(v) \) the number of vertices \( w\in N(v) \) such that $v$ is the leftmost neighbour of $w$, that is, $w=\ell_v$; see Figure~\ref{figure: alpha}. Analogously, denote by \( \alpha_r(v) \) the number of vertices \( w\in N(v) \) such that $v$ is the rightmost neighbour of $w$, that is, \( v=r_w \).
Since any vertex has exactly one leftmost edge and exactly one rightmost edge, we get 
\[
    \label{equation:sumofalpha}
    \sum_{ v\in V(G) }   \alpha_\ell(v)=\sum_{v\in V(G)}\alpha_r(v)=n.
\]
Also, we need the standard equality
\[
\label{equation:sumofdegrees}
\sum_{w\in V(G)}\deg w=2e.
\]
\begin{corollary}
\label{corollary:summation}
We have
\begin{gather}
\sum_{v\in V(G)} \Big( | DJ_\ell (v) | + | DJ_r (v) | \Big) \geq - 2(e -n) + \sum_{w\in V(G)} \deg^2 w 
\\ 
- 
\sum_{w\in V(G)}\big(\alpha_\ell(w) + \alpha_r(w)\big)\deg w
+ \sum_{v\in V(G)} |\mathcal L(v)| + \sum_{v\in V(G)}|\mathcal R(v)|
\end{gather}
\end{corollary}

\begin{proof}
Summing up the inequalities from Lemma~\ref{lemma:simple?} for all vertices of \(G\), we obtain that
\begin{gather}
\sum_{v\in V(G)} \Big( | DJ_\ell (v) | + | DJ_r (v) | \Big) \geq \sum_{v\in V(G)} \sum_{w\in N(v)} (\deg w - 1)
\\-\sum_{w\in V(G)} (\alpha_\ell(w)+\alpha_r(w))(\deg w -1)+\sum_{v\in V(G)} |\mathcal L(v)| + \sum_{v\in V(G)}|\mathcal R(v)|.
\end{gather}
By~\eqref{equation:sumofalpha} and~\eqref{equation:sumofdegrees}, we are done.
\end{proof}

For a vertex \(v\), let \(\mathcal L'(v)\) be a subset of \(\mathcal L(v)\) consisting of edges \(\ell_vw\) with \(\ell_w= \ell_v\), that is, \( \ell_v \) is the leftmost neighbour also for \( w \); see Figure~\ref{figure: alpha}. 
Analogously, let \(\mathcal R'(v)\) be a subset of \(\mathcal R(v)\) consisting of edges \(r_vw\) with \(r_w= r_v\), that is, \( r_v \) is the rightmost neighbour also for \( w \). 

\begin{lemma}
\label{lemma:discharging!}
For a vertex \(v\in V(G)\), the following equalities hold
\[
    \sum_{v\in V(G)}|\mathcal L'(v)| = \sum_{v\in V(G)}\frac{\alpha_\ell(v)(\alpha_\ell(v)-1)}{2}
\text{\ \ and\ \ }
    \sum_{v\in V(G)} |\mathcal R'(v)| = \sum_{v\in V(G)} \frac{\alpha_r(v)(\alpha_r(v)-1)}{2}.
\]
\end{lemma}

\begin{proof}
If we prove that for each vertex \( w\in V(G) \), we have
\[
\label{equality1}
    \sum_{v\in V(G): \ell_v=w} |\mathcal L'(v)|=\frac{\alpha_\ell(w)(\alpha_\ell(w)-1)}{2}
\]

then the first desired equality follows. Analogously, one can show the second equality.

Since \eqref{equality1} is trivial if \(\alpha_\ell(w)=0 \), we may assume \(\alpha_\ell(w)=k>0\). Let \( u_1, \dots, u_k\in N(w) \) be distinct vertices with \(\ell_{u_i}=w\); see the vertex \(w\) and its neighbours on Figure~\ref{figure: alpha}. Without loss of generality, assume that among \(wu_1, \dots, wu_k\), the edge \(wu_1\) is the leftmost edge, the edge \( wu_2 \) is the second leftmost edge, etc., that is, the angles \(\angle u_iwu_j\) for \( 1\leq i<j\leq k \) are positive. Therefore, \(|\mathcal L'(u_i)| = k-i\) for \(1\leq i\leq k\). Summing up all these equalities, we obtain~\eqref{equality1}. 
\end{proof}
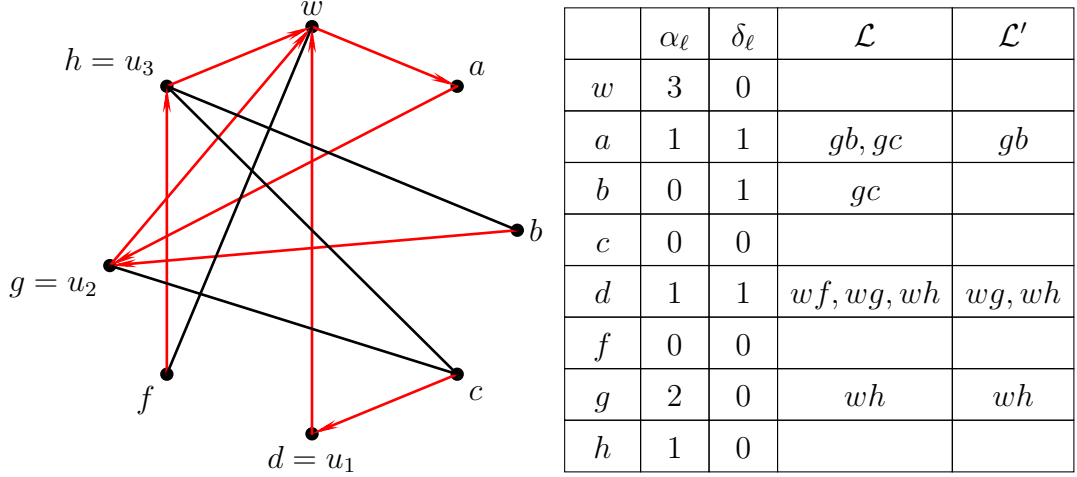
\begin{figure}[h]
    \begin{subfigure}[b]{0.45\textwidth}
	
		\begin{tikzpicture}[line cap=round,line width = 1pt, x = 0.9cm, y = 0.9cm]
		    \def \rad{3};
			\coordinate (b) at (0:\rad);
			\coordinate (c) at (-45:\rad);
			\coordinate (d) at (-90:\rad);
			\coordinate (e) at (-135:\rad);
			\coordinate (f) at (-170:\rad);
			\coordinate (g) at (-225:\rad);
			\coordinate (w) at (-270:\rad);
			\coordinate (a) at (-315:\rad);
			
			\filldraw (a) circle (2pt);
			\filldraw (b) circle (2pt);
			\filldraw (c) circle (2pt);
			\filldraw (d) circle (2pt);
			\filldraw (e) circle (2pt);
			\filldraw (f) circle (2pt);
			\filldraw (g) circle (2pt);
			\filldraw (w) circle (2pt);
			
			\draw (a) node  [above right]{\(a\)};
			\draw (b) node  [right]{\(b\)};
			\draw (c) node  [below right]{\(c\)};
			\draw (d) node  [below]{\(d = u_1\)};
			\draw (e) node  [below left]{\(f\)};
			\draw (f) node  [below left]{\(g = u_2\)};
			\draw (g) node  [above left]{\(h = u_3\)};
			\draw (w) node  [above]{\(w\)};
			
			\draw[red, -{Stealth[length=10pt, width=3pt]}] (w) -- (a);
			\draw[red, -{Stealth[length=10pt, width=3pt]}] (a) -- (f);
			\draw[red, -{Stealth[length=10pt, width=3pt]}] (b) -- (f);
			\draw[red, -{Stealth[length=10pt, width=3pt]}] (c) -- (d);
			\draw[red, -{Stealth[length=10pt, width=3pt]}] (d) -- (w);
			\draw[red, -{Stealth[length=10pt, width=3pt]}] (e) -- (g);
			\draw[red, -{Stealth[length=10pt, width=3pt]}] (f) -- (w);
			\draw[red, -{Stealth[length=10pt, width=3pt]}] (g) -- (w);
			
			\draw (g) -- (b);
			\draw (g) -- (c);
			\draw (e) -- (w);
			\draw (f) -- (c);

		\end{tikzpicture}
		\end{subfigure}
		\begin{subfigure}[b]{0.45\textwidth}
		\begin{tikzpicture}[cell/.style={rectangle,draw=black}, nodes in empty cells]
  \matrix[
  matrix of nodes,
  row sep =-\pgflinewidth,
  column sep = -\pgflinewidth,
  nodes={anchor=center,text height=2ex,text depth=0.25ex,cell},
  column 1/.style = {nodes={minimum width=1cm}},
  column 2/.style = {nodes={minimum width=0.9cm}},
  column 3/.style = {nodes={minimum width=0.9cm}},
  column 4/.style = {nodes={minimum width=2.3cm}},
  column 5/.style = {nodes={minimum width=1.6cm}},
  ] 
  {  & \(\alpha_\ell\) & \(\delta_\ell\) & \(\mathcal L\) & \(\mathcal L'\) \\
    \(w\) & 3 & 0 & \(\)          &\(\)      \\
    \(a\) & 1 & 1 & \(gb,gc\)     &\(gb\)    \\
    \(b\) & 0 & 1 & \(gc\)        &\(\)      \\
    \(c\) & 0 & 0 & \(\)          &\(\)      \\
    \(d\) & 1 & 1 & \(wf,wg,wh\)  &\(wg,wh\) \\
    \(f\) & 0 & 0 & \(\)          &\(\)      \\
    \(g\) & 2 & 0 & \(wh\)        &\(wh\)    \\
    \(h\) & 1 & 0 & \(\)          &\(\)      \\
  };
\end{tikzpicture}
		\end{subfigure}
	
	\caption{Here we use red arrows to illustrate leftmost edges: If an arrow connects \(x\) to \(y\), then \(y\) is the leftmost neighbour of \(x\).}
\label{figure: alpha}	
\end{figure}

To apply induction in Theorem~\ref{theorem:main2}, we consider the contribution of the vertices that satisfy one of the properties
\[
\deg(v)=\alpha_\ell(v), \deg(v)=\alpha_r(v), \mathcal L'(v)\subsetneq \mathcal L (v),\text{ or } \mathcal R'(v) \subsetneq \mathcal R(v).
\]
For that, we introduce the following auxiliary notation. 

Denote by \( n_\ell\) the number of vertices \( v\in V(G) \) such that for every \(w\in N(v)\) we have \( v=\ell_w \), that is, \( \deg(v) = \alpha_\ell(v) \). Analogously, denote by \( n_r \) the number of vertices \( v\in V(G) \) such that for every \( w\in N(v) \) we have \( v = r_w\), that is, \(\deg (v)=\alpha_r(v)\). 

For a vertex \(v\in V(G)\), put \(\delta_\ell(v)=1\) if there is at least one edge \(\ell_vx\) such that \(\angle x\ell_v v\) is positive and \(\ell_x\ne \ell_v\), otherwise, put \( \delta_\ell(v)=0 \); see Figure~\ref{figure: alpha}. Analogously, we define \(\delta_r(v)\). Remark that \(\delta_\ell(v)=1\) if and only if the set \( \mathcal L(v)\setminus\mathcal L'(v)\) is not empty. In particular
, we have
\[
\label{sum:ahaha}
    |\mathcal L(v)|\geq |\mathcal L'(v)|+\delta_\ell(v) \text{ and } |\mathcal R(v)|\geq |\mathcal R'(v)|+ \delta_r(v).
\]

\begin{corollary}
\label{maincorollary}
The following inequality holds
\begin{gather}
\sum_{v \in V(G)} |DJ_\ell(v)|+ \sum_{v \in V(G)}|DJ_r(v)|\\ \geq \left(\frac{(2e-n)^2}{2(n -  n_\ell)} - e + \frac{n}{2} + \sum_{v\in V(G)}\delta_\ell(v)\right) + \left(\frac{(2e-n)^2}{2(n - n_r)} - e + \frac{n}{2}+\sum_{v\in V(G)}\delta_r(v)\right).
\end{gather}
\end{corollary}
\begin{proof}
 The idea of the proof is to apply the inequality of arithmetic and geometric means to Corollary~\ref{corollary:summation}. Indeed, 
by Corollary~\ref{corollary:summation}, Lemma~\ref{lemma:discharging!} and \eqref{sum:ahaha}, we obtain
\begin{gather}
    \sum\limits_{v \in V(G)}\Big( |DJ_\ell(v)|+ |DJ_r(v)|\Big) \geq     \sum_{w\in V(G)} \Big( \frac{\deg w^2}{2}-\alpha_\ell(w)\deg w + \frac{\alpha_\ell(w)(\alpha_\ell(w)-1)}{2}+
    \\
    +\frac{\deg w^2}{2}-\alpha_r(w)\deg w + \frac{\alpha_r(w)(\alpha_r(w)-1)}{2}\Big)-2(e-n)+\sum_{v\in V(G)}\big(\delta_\ell(v)+\delta_r(v)\big)\label{lastsum}
\end{gather}
By~\eqref{equation:sumofalpha} and~\eqref{equation:sumofdegrees}, we have
\begin{gather}
    \sum_{w\in V(G)} \Big( \frac{\deg w^2}{2}-\alpha_\ell(w)\deg w + \frac{\alpha_\ell(w)(\alpha_\ell(w)-1)}{2}\Big)\\
= \sum_{w\in V(G)}\frac{(\deg w - \alpha_\ell(w))^2}{2} -\frac{n}{2}
    \geq \frac{(2e-n)^2}{2(n-n_\ell)} - \frac{n}{2},
\end{gather}
where to prove the last inequality, we use that the number of vanishing terms \(\deg w-\alpha_\ell(w)\) equals \(n_\ell\).

Analogously, we show that
\[
    \sum_{w\in V(G)} \Big( \frac{\deg w^2}{2}-\alpha_r(w)\deg w + \frac{\alpha_r(w)(\alpha_r(w)-1)}{2}\Big)\geq \frac{(2e-n)^2}{2(n-n_r)} - \frac{n}{2}.
\]
Substituting these two inequalities in \eqref{lastsum}, we finish the proof of the corollary.
\end{proof}

\begin{corollary}\label{theorem1corollary} The following inequality holds
\[
\sum\limits_{v \in V(G)}\Big( |DJ(v\ell_v)| + |DJ(vr_v)|\Big) \geq n\left(\frac{2e}{n}-1\right)\left(\frac{2e}{n}-2\right).
\]
\end{corollary}
\begin{proof}
By the definitions of \( DJ_r(v) \) and \( DJ_\ell(v)\), we have \( |DJ(v\ell_v)| \geq |DJ_\ell(v)| \) and \( |DJ(vr_v)| \geq |DJ_r(v)|\), and therefore, Corollary~\ref{maincorollary} yields
\begin{gather}
    \sum\limits_{v \in V(G)}\Big( |DJ(v\ell_v)| + |DJ(vr_v)|\Big) \geq
\frac{(2 e-n)^2}{n} - (2e - n)= n\left(\frac{2e}{n}-1\right)\left(\frac{2e}{n}-2\right),
\end{gather}
which finishes the proof.

\end{proof}
\section{Proof of Theorem~\ref{theorem:main1}}
\label{section:proof1}

Consider two possible cases. 

\textit{Case 1.} Let \( G \) be a pointed graph. Since \( (\sqrt{1+8m}+3)/4\geq 1\), we may assume that \( G \) does not contain vertices of degree 0 or 1; otherwise, we can delete such a vertex and use induction on the number of vertices. Hence we can apply the results of Section~\ref{section:preliminaries}. Combining Corollary~\ref{theorem1corollary} together with \( |DJ(v\ell_v)| \leq m\) and \( |DJ(vr_v)| \leq m \), we obtain
\[
    2mn \geq n\left(\frac{2e}{n}-1\right)\left(\frac{2e}{n}-2\right)=n\left(\left(\frac{2e}{n} -\frac32\right)^2 - \frac14\right),
\]
which finishes the proof of the first case.

\textit{Case 2.} Let \( G \) have a cyclic vertex \( v \). Hence, there are vertices \( v_1, v_2, v_3 \in N(v) \) such that \( v \) lies in the convex hull of \( v_1, v_2,\) and \( v_3 \). Remark that any edge that is not incident to \( v \) is disjoint from at least one of the edges \( vv_1, vv_2, vv_3 \). Since each of the edges \( vv_1, vv_2, \) and \( vv_3\) is disjoint from at most \( m \) edges, we have that there are at most \( 3m + \deg v \) edges in \( G \). The inequality \( \deg v < n \) finishes the proof of the second case.

\section{Proof of Theorem~\ref{theorem:main2}}
\label{section:proof2}

The proof is by induction on \( e \). Since we assume \( e \geq n/2\), we have \(d(G)\geq 1\).

\textit{Base of induction.} Assume that \( e \leq 3n/2 \), and thus, \( d(G)=2e/n\leq 3\).

If \(1\leq d(G) \leq 2\), the desired inequality trivially follows from \(\binom{d(G)}{3}\leq 0\).

If \(2<d(G)\leq 3\), then consider a graph \( G' \) obtained from \( G \) by deleting one edge in every pair of disjoint edges. Thus all edges in \(G'\) are pairwise intersecting. By Theorem~\ref{erdos}, we have \( |E(G')| \leq n \). Since we delete at most \(|DJ(G)|\) edges, we obtain
\[
    |DJ(G)| \geq  |E(G)| - |E(G')| \geq e-n = \frac{n(d(G)-2)}{2} \geq \frac{n}{2} \binom{d(G)}{3}.
\]

\textit{Induction step.} Assume that \( e > 3n/2 \), and thus, \(d(G)>3\).

First, we show that we may assume that \( G \) does not contain vertices of degree 0 or~1. Indeed, let 
\[
    F(G) := \frac{|V(G)|}{2} \binom{d(G)}{3}
\]
and suppose that \( G \) has a vertex of degree 0 or 1. Let \( G' \) be a graph obtained from \( G \) by removing this vertex. Then \(|DJ(G)| \geq |DJ(G')|\) and 
\begin{gather}
    F(G') \geq \frac{n-1}{2} \binom{\frac{2e-2}{n-1}}{3} =
    \frac{2e-2n}{12}  \left(\frac{2e-2}{n-1}\right)\left(\frac{2e-n-1}{n-1}\right) \\
    \overset{(*)}{>} 
    \frac{2e-2n}{12}  \left(\frac{2e-2}{n-1}\right)\left(\frac{2e-n-1}{n-1}\right) \left(\frac{e(n-1)}{(e-1)n}\right)\left(\frac{(2e-n)(n-1)}{(2e-n-1)n}\right)\\
    =\frac{2e(2e-n)(2e-2n)}{12n^2} = F(G).
\end{gather}
Here in \( (*) \) we use that 
\[
\frac{e(n-1)}{(e-1)n} < 1 \text{  and  } 
\frac{(2e-n)(n-1)}{(2e-n-1)n} < 1.
\]
Both of these inequalities easily follow from \( e > n > 1 \). Since \(d(G')>d(G)>3\), we may apply induction on the number of vertices for $G'$ and obtain the desired inequality
\[
    |DJ(G)|\geq |DJ(G')|\geq F(G')\geq F(G).
\]
Therefore, without loss of generality we assume that each vertex of \( G \) has degree at least 2, and thus, we can use the results of Section~\ref{section:preliminaries}.
\bigskip

Finally, we are almost ready to delete from \( G \) either all leftmost or all rightmost edges and then apply the induction hypothesis to the obtained graph.

By  Corollary~\ref{maincorollary}, we may assume without loss of generality that
\[
\label{equation of corollary}
\sum\limits_{v \in V(G)}\ncrlsetcard{v} \geq \frac{(2 e-n)^2}{2(n -  n_\ell) } -e +\frac{n}{2} + \sum_{v\in V(G)}\delta_\ell(v).
\]
Otherwise, if this inequality does not hold, then Corollary~\ref{maincorollary} yields the analogous inequality for the $r$-summands.

Let \(G'\) be a graph obtained from \(G\) by deleting all leftmost edges and also all vertices \(v\in V(G)\) with \(\alpha_\ell(v)=\deg v\) (as they become vertices of degree 0); see Figure~\ref{figure: delete}. Let us find the numbers of vertices and edges in the new graph. Since \(n_\ell \) is the number of deleted vertices, we have 
\[
    |V(G')|=n-n_\ell.
\] 

\begin{figure}[h]
    \begin{subfigure}[b]{0.45\textwidth}
	
		\begin{tikzpicture}[line cap=round,line width = 1pt, x = 1cm, y = 1cm]
		    \def \rad{3};
			\coordinate (b) at (0:\rad);
			\coordinate (c) at (-36:\rad);
			\coordinate (d) at (-72:\rad);
			\coordinate (e) at (-108:\rad);
			\coordinate (f) at (-144:\rad);
			\coordinate (g) at (-180:\rad);
			\coordinate (h) at (-216:\rad);
			\coordinate (i) at (-245:\rad);
			\coordinate (j) at (-288:\rad);
			\coordinate (a) at (-324:\rad);
			
			\filldraw (a) circle (2pt);
			\filldraw (b) circle (2pt);
			\filldraw (c) circle (2pt);
			\filldraw (d) circle (2pt);
			\filldraw (e) circle (2pt);
			\filldraw (f) circle (2pt);
			\filldraw (g) circle (2pt);
			\filldraw (h) circle (2pt);
			\filldraw (i) circle (2pt);
			\filldraw (j) circle (2pt);
			
			\draw (a) node  [above right]{\(a\)};
			\draw (b) node  [right]{\(b\)};
			\draw (c) node  [below right]{\(c\)};
			\draw (d) node  [below]{\(d\)};
			\draw (e) node  [below left]{\(f\)};
			\draw (f) node  [below left]{\(g\)};
			\draw (g) node  [above left]{\(h\)};
			\draw (h) node  [above]{\(i\)};
			\draw (i) node  [above]{\(j\)};
			\draw (j) node  [above]{\(k\)};
			
			\draw [red, -{Stealth[length=10pt, width=3pt]}](a) -- (c);
			\draw [red, -{Stealth[length=10pt, width=3pt]}](e) -- (a);
			\draw [red, -{Stealth[length=10pt, width=3pt]}](g) -- (a);
			\draw [red, -{Stealth[length=10pt, width=3pt]}](j) -- (b);
			\draw [red, dotted,{Stealth[length=10pt, width=3pt]}-{Stealth[length=10pt, width=3pt]}](b) -- (h);
			\draw [red, -{Stealth[length=10pt, width=3pt]}](c) -- (h);
			\draw [red, -{Stealth[length=10pt, width=3pt]}](d) -- (e);
			\draw [red, -{Stealth[length=10pt, width=3pt]}](f) -- (g);
			\draw [red, -{Stealth[length=10pt, width=3pt]}](i) -- (j);
			
			\draw (b) -- (i);
			\draw (d) -- (f);
			\draw (d) -- (g);
			\draw (a) -- (f);
			\draw (f) -- (j);
			\draw (c) -- (j);
			\draw (c) -- (i);

		\end{tikzpicture}
		\caption{Graph \(G\)}
		\end{subfigure}
		\begin{subfigure}[b]{0.45\textwidth}
			\begin{tikzpicture}[line cap=round,line width = 1pt, x = 1cm, y = 1cm]
		    \def \rad{3};
			\coordinate (b) at (0:\rad);
			\coordinate (c) at (-36:\rad);
			\coordinate (d) at (-72:\rad);
			\coordinate (e) at (-108:\rad);
			\coordinate (f) at (-144:\rad);
			\coordinate (g) at (-180:\rad);
			\coordinate (h) at (-216:\rad);
			\coordinate (i) at (-245:\rad);
			\coordinate (j) at (-288:\rad);
			\coordinate (a) at (-324:\rad);
			
			\filldraw (a) circle (2pt);
			\filldraw (b) circle (2pt);
			\filldraw (c) circle (2pt);
			\filldraw (d) circle (2pt);
			\filldraw (e) circle (2pt);
			\filldraw (f) circle (2pt);
			\filldraw (g) circle (2pt);
			\filldraw (i) circle (2pt);
			\filldraw (j) circle (2pt);
			
			\draw (a) node  [above right]{\(a\)};
			\draw (b) node  [right]{\(b\)};
			\draw (c) node  [below right]{\(c\)};
			\draw (d) node  [below]{\(d\)};
			\draw (e) node  [below left]{\(f\)};
			\draw (f) node  [below left]{\(g\)};
			\draw (g) node  [above left]{\(h\)};
			\draw (i) node  [above]{\(j\)};
			\draw (j) node  [above]{\(k\)};

		    \draw (b) -- (i);
			\draw (d) -- (f);
			\draw (d) -- (g);
			\draw (a) -- (f);
			\draw (f) -- (j);
			\draw (c) -- (j);
			\draw (c) -- (i);
			\end{tikzpicture}
			\caption{Graph \(G'\)}
		\end{subfigure}
	
	\caption{Here the leftmost edges are drawn as red arrows: If the red arrow connects $x$ to $y$, then $y$ is the leftmost neighbour of $x$. The dashed edge $ib$ is a double leftmost edge. Remark that we delete the vertex $i$ and at the same time do not delete~$f$ because it is not the leftmost neighbour for $a$.}
\label{figure: delete}	
\end{figure}
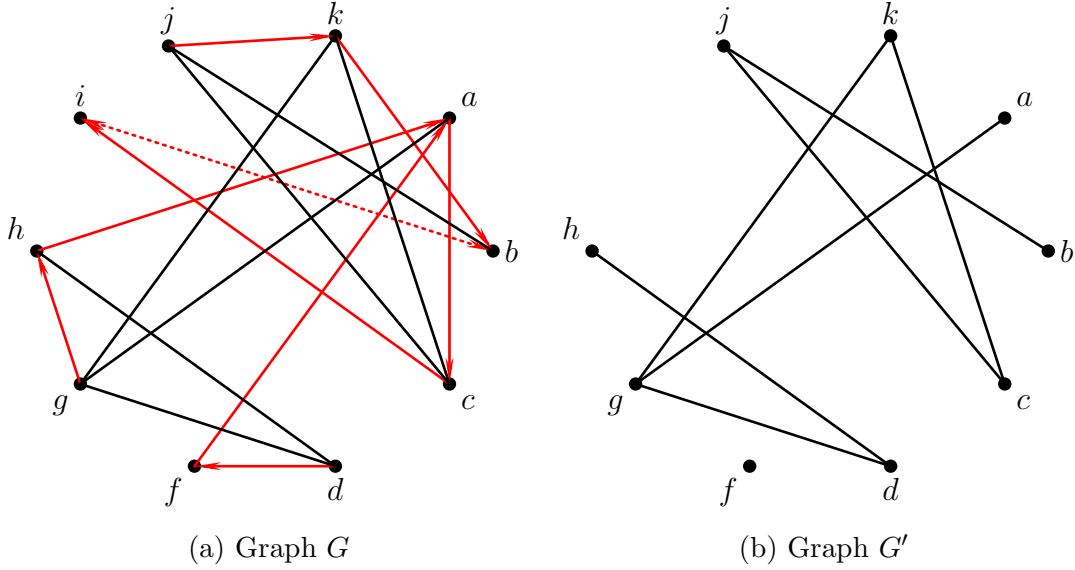

Denote by $t_\ell$ the number of \textit{double leftmost edges}, where an edge \(uv\) is called \textit{double leftmost} if leftmost with respect to both its endpoints \(u\) and \(v\), that is, \(u=\ell_v\) and \(v=\ell_u\). Since we delete each leftmost edge (and can not delete such an edge twice), we obtain 
\[
|E(G')|=e-n+t_\ell.
\]
\medskip 

First, we claim that the following inequality holds
\begin{equation}
\label{equation:djg}
|DJ(G)| \geq |DJ(G')| + \sum\limits_{v \in  V(G)} |DJ_\ell(v)| - |E(G')|.
\end{equation}
To prove this inequality, we show that each pair of disjoint edges in $G$ counted on the right-hand side of this inequality at most once. Indeed, there are three types of disjoint pairs of edges in $G$. 
\begin{itemize}
    \item[1.] Both edges are not leftmost. Then this pair belongs to $DJ(G')$.
    \item[2.] One of the edges is leftmost and the second one is not. Then we count such a pair at most once in the sum $\sum_{v\in V(G)}|DJ_\ell(v)|$. Recall that by $DJ_\ell(v)$, we denote edges incident to $N(v)$ and disjoint from $v\ell_v$.
    \item[3.] Both edges are leftmost. Denote them by $v\ell_v$ and $u\ell_u$. We can count their pair in the sum $\sum_{v\in V(G)}|DJ_\ell(v)|$ only in the case when $u$ and $v$ are neighbours in $G$ and the edge $uv\in E(G)$ is distinct from them.  Under this assumptions, we count this pair twice when consider the summands $|DJ_\ell(v)|$ and $|DJ_\ell(u)|$. However, in this case the edge $uv$ is not leftmost for $G$ and thus belongs to $G'$. As each edge $xy\in E(G')$ corresponds to such a pair of disjoint edges in $G$ and we subtract $|E(G')|$ on the right-hand side, we can say that any pair of disjoint leftmost edges is counted at most one there.
\end{itemize}
\medskip 

Second, for any double leftmost edge \(uv\in E(G)\) we have
\begin{equation}
\label{equality for double left edges}
\delta_\ell(v)+\beta_\ell(u)= 1,
\end{equation}
where for every \(w\in V(G)\) put \(\beta_\ell(w)=1\) if \(\deg w = \alpha_\ell(w)\), otherwise, \(\beta_\ell(w)=0\). Remark that if $\beta_\ell(w)=1$ then the edge $w\ell_w$ is a double leftmost edge. Recall that $\delta_\ell(w)=1$ if there is at least one edge $\ell_x y\in E(G)$ such that $\ell_x\ne \ell_y$ and the ray $\ell_x y$ shares with the affine cone $x+\cone\{\ell_x-x, r_x-x\}$ only the point $\ell_x$.  Otherwise, $\delta_\ell(x)=0$. 

To prove~(\ref{equality for double left edges}), we consider two possible cases. \begin{itemize}
    \item[1.] The vertex \(u\) is the leftmost vertex with respect to all its neighbours, and hence, $\beta_\ell(u)=1$ and  \(\delta_\ell(v)=0\).
    \item [2.] There is an edge $uw$ that is not leftmost with respect to $w$, and hence, $\beta_\ell(u)=0$ and $\delta_\ell(v)=1$.
\end{itemize}

By~\eqref{equality for double left edges}, we easily obtain that the equality
\begin{equation}
    \label{equation strong for double left edges}
    \delta_\ell(u)+\delta_\ell(v)+\beta_\ell(u)+\beta_\ell(v)= 2
\end{equation}
holds for all double leftmost edges $uv$. 

Recall that $n_{\ell}$ is the number of vertices \(w\) with \( \deg (w)= \alpha_\ell(w)\), that is, $n_\ell=\sum_{v\in V(G)} \beta_\ell(v)$.
Summing up \eqref{equation strong for double left edges} over all double leftmost edges in \(G\) and using that each vertex can be incident to at most one double leftmost edge, we easily obtain 
\[
\label{equation:the new inequality}
    \sum_{v\in V(G)}\delta_\ell(v)+\sum_{v\in V(G)}\beta_\ell(v)= \sum_{v\in V(G)}\delta_\ell(v)+n_\ell \geq 2t_\ell\geq n_\ell=\sum_{v\in V(G)} \beta_{\ell}(v),
\]
where $t_\ell$ is the number of double leftmost edges.
\medskip

Third, by~\eqref{equation:the new inequality}, we have $n-2t_\ell\leq n-n_\ell$, and thus, 
\[
\label{equation:newavarage}
    d(G') = \frac{2(e-n+t_\ell)}{n -  n_\ell}\geq \frac{2e-n}{n-n_\ell}-1\geq \frac{2e-n}{n} -1 =d(G)-2> 1,
\]
and thus, we can apply the induction hypothesis for \(G'\).
\medskip

By the induction hypothesis, \eqref{equation:djg}, \eqref{equation:newavarage}, and \eqref{equation of corollary}, we obtain
\[
|DJ(G)|\geq \frac{(n -  n_\ell)}{2} \cdot\binom{\frac{2e-n}{n-n_\ell} - 1}{3} + \frac{(2e-n)^2}{2(n -  n_\ell)} -e+\frac{n}{2}+\sum_{v\in V(G)}\delta_\ell(v)- e + n - t_\ell.
\]

Finally, by~\eqref{equation:the new inequality}, we have 
\[
    \sum_{v\in V(G)} \delta_\ell(v)-t_\ell\geq \frac{1}{2}\sum_{v\in V(G)} \delta_\ell(v)-t_\ell\geq \frac{n_\ell}{2},
\]
and thus,
\begin{align}
|DJ(G)|&\geq 
\frac{n -  n_\ell}{2} \cdot\binom{\frac{2e-n}{n-n_\ell} - 1}{3} + \frac{(2e-n)^2}{2(n -  n_\ell) } - (2e-n) + \frac{n-n_\ell}{2}
\\ 
&=\frac{n-n_\ell}{2}\cdot \Bigg(\binom{\frac{2e-n}{n-n_\ell}-1}{3}+\left(\frac{2e-n}{n-n_\ell}-1\right)^2\Bigg)\\
&=\frac{2e-n}{12}\cdot \Bigg(\left(\frac{2e-n}{n-n_\ell}\right)^2-1\Bigg)\\
&\geq
\frac{2e-n}{12}\cdot \bigg(\left(\frac{2e-n}{n}\right)^2-1\bigg)=\frac{n}{2}\binom{\frac{2e}{n}}{3},
\end{align}
which finishes the proof of the theorem.

\section{Discussion}
\label{section:discussion}

\subsection{Tightness of the theorems.} 
\label{subsection:tightness}
Let \( n, k \) be integers of different parity such that \( n-2 > k>2 \). Let \(G_{n,k}\) be a convex graph whose vertices are denoted \(x_1,\dots, x_n\) in cyclic order and edges are \(x_{i}x_{j}\) for \(j-i\equiv\frac{n-k-1}{2},\dots,\frac{n+k+1}{2} (\mathrm{mod}\ n)\); see Figure \ref{figure:example}. 

It is not difficult to verify that the number of edges of \(G_{n,k}\) is \(n(k+2)/2\). Moreover,
 \[
 |DJ(G_{n,k})| = \frac{n}{2} \cdot \binom{k+2}{3} \text{ and }
 |DJ(uv)| \leq \frac{k(k+1)}{2} \text{ for any edge } uv\in E(G_{n,k}).
 \]
Therefore, for the graph \(G_{n,k}\), the bounds in Theorems~\ref{theorem:main1} and~\ref{theorem:main2} are tight. 

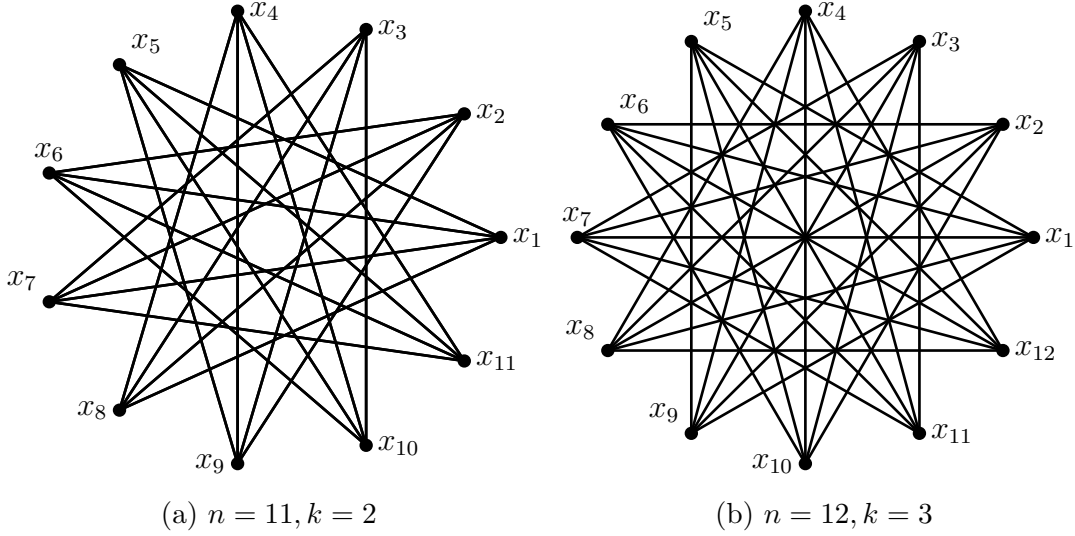
\begin{figure}[h]

\begin{subfigure}[b]{0.45\textwidth}
\begin{tikzpicture}[line cap=round,line width = 1pt, x = 3.03085cm, y = 3.03085cm]
\draw (1.0, 0.0) node  [right]{\(x_1\)};
\draw (0.8412535328311812, 0.5406408174555976)  node  [right]{\( x_2 \)};
\draw (0.41541501300188644, 0.9096319953545183)  node  [right]{\( x_3 \)};
\draw (-0.142314838273285, 0.9898214418809328)  node  [right]{\( x_4 \)};
\draw (-0.654860733945285, 0.7557495743542583)  node  [above right]{\( x_5 \)};
\draw (-0.9594929736144974, 0.28173255684142967)  node  [above]{\( x_6 \)};
\draw (-0.9594929736144975, -0.2817325568414294)  node  [above left]{\( x_7 \)};
\draw (-0.6548607339452852, -0.7557495743542582)  node  [left]{\( x_8 \)};
\draw (-0.14231483827328523, -0.9898214418809327)  node  [left]{\( x_9 \)};
\draw (0.41541501300188605, -0.9096319953545186)  node  [right]{\(x_{10}\)};
\draw (0.8412535328311812, -0.5406408174555974)  node  [right]{\( x_{11}\)};

\filldraw (1.0, 0.0) circle (2pt);
\filldraw (0.8412535328311812, 0.5406408174555976)  circle (2pt);
\filldraw (0.41541501300188644, 0.9096319953545183) circle (2pt);
\filldraw (-0.142314838273285, 0.9898214418809328)  circle (2pt);
\filldraw (-0.654860733945285, 0.7557495743542583)  circle (2pt);
\filldraw (-0.9594929736144974, 0.28173255684142967)  circle (2pt);
\filldraw (-0.9594929736144975, -0.2817325568414294)  circle (2pt);
\filldraw (-0.6548607339452852, -0.7557495743542582)  circle (2pt);
\filldraw (-0.14231483827328523, -0.9898214418809327)  circle (2pt);
\filldraw (0.41541501300188605, -0.9096319953545186)  circle (2pt);
\filldraw (0.8412535328311812, -0.5406408174555974)  circle (2pt);

\draw (1.0, 0.0)--(-0.654860733945285, 0.7557495743542583);
\draw (1.0, 0.0)--(-0.9594929736144974, 0.28173255684142967);
\draw (1.0, 0.0)--(-0.9594929736144975, -0.2817325568414294);
\draw (1.0, 0.0)--(-0.6548607339452852, -0.7557495743542582);
\draw (0.8412535328311812, 0.5406408174555976)--(-0.9594929736144974, 0.28173255684142967);
\draw (0.8412535328311812, 0.5406408174555976)--(-0.9594929736144975, -0.2817325568414294);
\draw (0.8412535328311812, 0.5406408174555976)--(-0.6548607339452852, -0.7557495743542582);
\draw (0.8412535328311812, 0.5406408174555976)--(-0.14231483827328523, -0.9898214418809327);
\draw (0.41541501300188644, 0.9096319953545183)--(-0.9594929736144975, -0.2817325568414294);
\draw (0.41541501300188644, 0.9096319953545183)--(-0.6548607339452852, -0.7557495743542582);
\draw (0.41541501300188644, 0.9096319953545183)--(-0.14231483827328523, -0.9898214418809327);
\draw (0.41541501300188644, 0.9096319953545183)--(0.41541501300188605, -0.9096319953545186);
\draw (-0.142314838273285, 0.9898214418809328)--(-0.6548607339452852, -0.7557495743542582);
\draw (-0.142314838273285, 0.9898214418809328)--(-0.14231483827328523, -0.9898214418809327);
\draw (-0.142314838273285, 0.9898214418809328)--(0.41541501300188605, -0.9096319953545186);
\draw (-0.142314838273285, 0.9898214418809328)--(0.8412535328311812, -0.5406408174555974);
\draw (-0.654860733945285, 0.7557495743542583)--(1.0, 0.0);
\draw (-0.654860733945285, 0.7557495743542583)--(-0.14231483827328523, -0.9898214418809327);
\draw (-0.654860733945285, 0.7557495743542583)--(0.41541501300188605, -0.9096319953545186);
\draw (-0.654860733945285, 0.7557495743542583)--(0.8412535328311812, -0.5406408174555974);
\draw (-0.9594929736144974, 0.28173255684142967)--(1.0, 0.0);
\draw (-0.9594929736144974, 0.28173255684142967)--(0.8412535328311812, 0.5406408174555976);
\draw (-0.9594929736144974, 0.28173255684142967)--(0.41541501300188605, -0.9096319953545186);
\draw (-0.9594929736144974, 0.28173255684142967)--(0.8412535328311812, -0.5406408174555974);
\draw (-0.9594929736144975, -0.2817325568414294)--(1.0, 0.0);
\draw (-0.9594929736144975, -0.2817325568414294)--(0.8412535328311812, 0.5406408174555976);
\draw (-0.9594929736144975, -0.2817325568414294)--(0.41541501300188644, 0.9096319953545183);
\draw (-0.9594929736144975, -0.2817325568414294)--(0.8412535328311812, -0.5406408174555974);
\draw (-0.6548607339452852, -0.7557495743542582)--(1.0, 0.0);
\draw (-0.6548607339452852, -0.7557495743542582)--(0.8412535328311812, 0.5406408174555976);
\draw (-0.6548607339452852, -0.7557495743542582)--(0.41541501300188644, 0.9096319953545183);
\draw (-0.6548607339452852, -0.7557495743542582)--(-0.142314838273285, 0.9898214418809328);
\draw (-0.14231483827328523, -0.9898214418809327)--(0.8412535328311812, 0.5406408174555976);
\draw (-0.14231483827328523, -0.9898214418809327)--(0.41541501300188644, 0.9096319953545183);
\draw (-0.14231483827328523, -0.9898214418809327)--(-0.142314838273285, 0.9898214418809328);
\draw (-0.14231483827328523, -0.9898214418809327)--(-0.654860733945285, 0.7557495743542583);
\draw (0.41541501300188605, -0.9096319953545186)--(0.41541501300188644, 0.9096319953545183);
\draw (0.41541501300188605, -0.9096319953545186)--(-0.142314838273285, 0.9898214418809328);
\draw (0.41541501300188605, -0.9096319953545186)--(-0.654860733945285, 0.7557495743542583);
\draw (0.41541501300188605, -0.9096319953545186)--(-0.9594929736144974, 0.28173255684142967);
\draw (0.8412535328311812, -0.5406408174555974)--(-0.142314838273285, 0.9898214418809328);
\draw (0.8412535328311812, -0.5406408174555974)--(-0.654860733945285, 0.7557495743542583);
\draw (0.8412535328311812, -0.5406408174555974)--(-0.9594929736144974, 0.28173255684142967);
\draw (0.8412535328311812, -0.5406408174555974)--(-0.9594929736144975, -0.2817325568414294);
\end{tikzpicture}
\subcaption{\( n = 11, k = 2\)}
\end{subfigure}
\begin{subfigure}[b]{0.45\textwidth}
\begin{tikzpicture}[line cap=round,line width = 1pt, x=3cm,y=3cm]

\draw (1.0, 0.0) node [right] {\( x_1 \)};
\draw (0.8660254038, 0.5) node [right] {\( x_2\)};
\draw (0.5, 0.8660254038) node [right] {\(x_3\)};
\draw (0.0, 1.0) node [right] {\( x_4 \)};
\draw (-0.5, 0.8660254038) node [above right] {\( x_5 \)};
\draw (-0.8660254038, 0.5) node [above right] {\( x_6 \)};
\draw (-1.0, 0.0) node [above] {\( x_7 \)};
\draw (-0.8660254038, -0.5) node [above left] {\( x_8 \)};
\draw (-0.5, -0.8660254038) node [above left] {\( x_9 \)};
\draw (0.0, -1.0) node [left] {\( x_{10} \)};
\draw (0.5, -0.8660254038) node [right] {\( x_{11} \)};
\draw (0.8660254038, -0.5) node [right] {\( x_{12} \)};

\filldraw (1.0, 0.0) circle (2pt);
\filldraw (0.8660254038, 0.5) circle (2pt);
\filldraw (0.5, 0.8660254038) circle (2pt);
\filldraw (0.0, 1.0) circle (2pt);
\filldraw (-0.5, 0.8660254038) circle (2pt);
\filldraw (-0.8660254038, 0.5) circle (2pt);
\filldraw (-1.0, 0.0) circle (2pt);
\filldraw (-0.8660254038, -0.5) circle (2pt);
\filldraw (-0.5, -0.8660254038) circle (2pt);
\filldraw (0.0, -1.0) circle (2pt);
\filldraw (0.5, -0.8660254038) circle (2pt);
\filldraw (0.8660254038, -0.5) circle (2pt);

\draw (1.0 ,  0.0 ) -- ( -0.5 ,  0.8660254038 );
\draw ( 1.0 ,  0.0 ) -- ( -0.8660254038 ,  0.5 );
\draw ( 1.0 ,  0.0 ) -- ( -1.0 ,  0.0 );
\draw ( 1.0 ,  0.0 ) -- ( -0.8660254038 ,  -0.5 );
\draw ( 1.0 ,  0.0 ) -- ( -0.5 ,  -0.8660254038 );
\draw ( 0.8660254038 ,  0.5 ) -- ( -0.8660254038 ,  0.5 );
\draw ( 0.8660254038 ,  0.5 ) -- ( -1.0 ,  0.0 );
\draw ( 0.8660254038 ,  0.5 ) -- ( -0.8660254038 ,  -0.5 );
\draw ( 0.8660254038 ,  0.5 ) -- ( -0.5 ,  -0.8660254038 );
\draw ( 0.8660254038 ,  0.5 ) -- ( 0.0 ,  -1.0 );
\draw ( 0.5 ,  0.8660254038 ) -- ( -1.0 ,  0.0 );
\draw ( 0.5 ,  0.8660254038 ) -- ( -0.8660254038 ,  -0.5 );
\draw ( 0.5 ,  0.8660254038 ) -- ( -0.5 ,  -0.8660254038 );
\draw ( 0.5 ,  0.8660254038 ) -- ( 0.0 ,  -1.0 );
\draw ( 0.5 ,  0.8660254038 ) -- ( 0.5 ,  -0.8660254038 );
\draw ( 0.0 ,  1.0 ) -- ( -0.8660254038 ,  -0.5 );
\draw ( 0.0 ,  1.0 ) -- ( -0.5 ,  -0.8660254038 );
\draw ( 0.0 ,  1.0 ) -- ( 0.0 ,  -1.0 );
\draw ( 0.0 ,  1.0 ) -- ( 0.5 ,  -0.8660254038 );
\draw ( 0.0 ,  1.0 ) -- ( 0.8660254038 ,  -0.5 );
\draw ( -0.5 ,  0.8660254038 ) -- ( -0.5 ,  -0.8660254038 );
\draw ( -0.5 ,  0.8660254038 ) -- ( 0.0 ,  -1.0 );
\draw ( -0.5 ,  0.8660254038 ) -- ( 0.5 ,  -0.8660254038 );
\draw ( -0.5 ,  0.8660254038 ) -- ( 0.8660254038 ,  -0.5 );
\draw ( -0.8660254038 ,  0.5 ) -- ( 0.0 ,  -1.0 );
\draw ( -0.8660254038 ,  0.5 ) -- ( 0.5 ,  -0.8660254038 );
\draw ( -0.8660254038 ,  0.5 ) -- ( 0.8660254038 ,  -0.5 );
\draw ( -1.0 ,  0.0 ) -- ( 0.5 ,  -0.8660254038 );
\draw ( -1.0 ,  0.0 ) -- ( 0.8660254038 ,  -0.5 );
\draw ( -0.8660254038 ,  -0.5 ) -- ( 0.8660254038 ,  -0.5 );

\end{tikzpicture}
\subcaption{\( n = 12, k = 3 \)}
\end{subfigure}
\caption{Graph \(G_{n,k}\).}
\label{figure:example}
\end{figure}
\subsection{Open problems and conjectures.}
Unfortunately, Theorem~\ref{theorem:main2} becomes wrong if one replaces the bound
\[
\frac{n}{2}\binom{d(G)}{3}\quad \text{by}\quad
\frac{1}{2}\sum_{v \in V(G)}\binom{\deg v}{3}.
\]
Indeed, a star \(S\) on \( n + 1 \) vertices satisfies the following inequality
\[
\frac{1}{2}\sum_{v \in V(S)}\binom{\deg v}{3} = \binom{n}{3} > 0 = DJ(S).
\]

We believe that the following conjectures hold.
\begin{conjecture}
\label{conjecture:main}
Let \(m\) be a non-negative integer and \(G\) be a geometric graph such that \( |DJ(uv)|\leq m\) for any edge \( uv\in E(G) \). Then
\[ 
    |E(G)|\leq \frac{\sqrt{1+8m}+3}{4}\cdot |V(G)|.
\]
\end{conjecture}
\begin{conjecture}
For any geometric graph \( G \) with \(2|E(G)|\geq |V(G)|\), we have
\[
|DJ(G)| \geq \frac{n}{2} \cdot \binom{d(G)}{3}.
\]
\end{conjecture}
At last, remark that Theorem~\ref{theorem:main1} implies Conjecture~\ref{conjecture:main} for a geometric graph \( G\) with at least \(\frac{3(\sqrt{1+8m}+1)}{2}\) vertices.


\bibliographystyle{amsalpha}
\bibliography{biblio.bib}

\end{document}